\documentclass[english,10pt]{article}
\usepackage[english]{babel}
\usepackage{amsmath,amsthm,amssymb}
\usepackage{latexsym}
\usepackage{enumerate}
\usepackage[pdftex,plainpages=false,pdfpagelabels,colorlinks=true,citecolor=black,linkcolor=black,urlcolor=black,filecolor=black,bookmarksopen=true]{hyperref}

\newtheorem{theorem}{Theorem}
\newtheorem{lemma}{Lemma}
\newtheorem{proposition}{Proposition}

\newtheorem{remark}{Remark}

\title{A Large deviations principle at zero temperature for stationary Markov equilibrium states on countable Markov shifts}

\author{Victor Vargas\thanks{Center for Mathematics at the University of Porto.  Supported by FCT through the project UIDP/00144/2020}}

\begin{document}

\maketitle

\begin{abstract}
Consider a topologically transitive unilateral countable Markov shift $\Sigma$, a locally constant potential $\phi : \Sigma \to \mathbb{R}$ satisfying suitable conditions, and assume that $\mu_t$ is the unique stationary Markov equilibrium state associated to the potential $t\phi$ for each $t \geq 1$.  In this paper we prove a first level large deviations principle at zero temperature for the family of equilibrium states $(\mu_t)_{t \geq 1}$ and we extend the result to the setting of bilateral countable Markov shifts.
\end{abstract}

\vspace{2mm}

{\footnotesize {\bf Keywords:} Countable Markov shift, equilibrium state, large deviations, stationary Markov measure, zero temperature limit}

\vspace*{2mm}

{\footnotesize {\bf Mathematics Subject Classification (2020):}  28Dxx, 37Axx, 60Jxx.}

\vspace*{2mm}

\section{Introduction}

The thermodynamic formalism on the context of countable Markov shifts have been widely studied by several authors.  For instance,  in \cite{MR1738951} they were presented basic results in this setting,  like the existence of equilibrium states via Ruelle-Perron-Frobenius theorem and some interesting properties of the Gurevich pressure besides a variational principle.  Some time later,  in \cite{MR2151222} it was studied the behavior of a family of equilibrium states associated to a suitable potential guaranteeing the existence of ground states when the countable Markov shift satisfies the so-called BIP property.  Later,  in \cite{MR2800665} it was showed the uniqueness of such a ground state in the context of locally constant potentials.  Furthermore,  the former results were generalized to the setting of topologically transitive countable Markov shifts assuming similar conditions on the potential in \cite{MR3864383} and \cite{MR4637152}.

In this paper we present a large deviation principle at zero temperature in the context of topologically transitive countable Markov shifts as a generalization to a more general dynamical setting of results appearing in \cite{MR3114331},  \cite{MR3779019} and \cite{BMP15}.  The mentioned result is proved following similar approaches to the ones proposed in \cite{MR2210682} and \cite{MR2496111} using a couple of techniques well-known in the mathematical literature as calibrated sub-actions and involution kernels (see \cite{MR1855838},  \cite{MR1841880} and \cite{MR1827117} for seminal approaches).  

In fact,  the so-called involution kernel has been useful as a tool to prove existence and uniqueness of Gibbs states and characterize the eigenfunctions associated to the Ruelle-Perron-Frobenius operator on interesting classes of dynamical systems via duality (see for instance \cite{MR3656287},  \cite{MR3377291} and \cite{MR1827117}).  Furthermore,  in \cite{MR2210682},  \cite{MR2496111} and \cite{zbMATH07528592},  the authors use the involution kernel to prove large deviation principles at zero temperature on interesting dynamical approaches.  

On the other hand,  the so-called sub-actions appear as dual objects of the maximizing measures associated to a suitable potential,  which become very useful at the moment of characterizing the support of the mentioned measures (see for instance \cite{MR2563132} and \cite{MR2765475}).  Moreover,  several properties of the so-called maximizing measures on the symbolic dynamics approach can be found in the works \cite{MR2563132},  \cite{MR2191393} and \cite{MR3529118}.

In here,  we are able to prove a large deviation principle using the characterization of the equilibrium state associated to a locally constant potential as a Markov stationary probability measure,  where the so-called rate function is given in terms of the maximizing value of the potential and a sub-action of the potential which can be obtained via Ruelle-Perron-Frobenius theory.  Besides that,  we extend that large deviation principle to a bilateral approach using as main tool the involution kernels associated to the potential.  

This paper is organized as follows: In section \ref{mr-section} there are presented the basic definitions related with the theory which we are interested in and there are stated the main results to be proved in this work.  In section \ref{LDP-section} appear the proofs of the main results of this paper,  the first one related with the large deviations principle on unilateral countable Markov shifts appears in section \ref{ULDP-section} and in section \ref{BLDP-section} is presented the extension to the bilateral context.

\section{Statement of the main result}
\label{mr-section}

Consider the {\bf alphabet} $S := \mathbb{N}$ and the {\bf space of sequences} $X := S^{\mathbb{N}_0}$ equipped with the {\bf product topology},  where $\mathbb{N}_0 := \mathbb{N} \cup \{0\}$.  It is not difficult to check that such a topology on $X$ can be induced by {\bf cylinders}, i.e., sets of the form $[\omega] := \{x \in X :\; x_0 = \omega_0, ... , x_n = \omega_n\}$,  with $|\omega| := n+1$ the length of the cylinder.  Furthermore,  it is easy to check that the {\bf metric} $d(x, y) := 2^{-m(x, y)}$, where $m(x, y) := \min\{n \in \mathbb{N}_0:\; x_n \neq y_n\}$, induces the structure of a {\bf metric space} on $X$.  The {\bf shift map} $\sigma : X \to X$ is defined by $\sigma((x_n)_{n \in \mathbb{N}_0}) := (x_{n + 1})_{n \in \mathbb{N}_0}$. 

Given a matrix $M = (m_{ij})_{S \times S}$ of zeroes and ones such that no column or row of $M$ is composed only by zeroes,  we define the {\bf countable Markov shift} $\Sigma$ with {\bf incidence matrix} $M$ by $\Sigma := \{x \in X :\; m_{x_i x_{i+1}} = 1\; \forall i \in \mathbb{N}_0 \}$, with the shift map $\sigma$ acting on it. 

Consider a {\bf continuous function} $\phi : \Sigma \to \mathbb{R}$, we say that $\phi$ is a {\bf locally constant} potential when there is $k \in \mathbb{N}_0$ such that $\phi(x) = \phi(x_0\; ...\; x_k)$ for each $x \in \Sigma$. Further, $\phi$ is called a {\bf Markov potential} when $\phi(x) = \phi(x_0 x_1)$ and in that case we are able to define the {\bf transpose potential} of $\phi$ by the expression $\phi^\intercal(x_0 x_1) = \phi(x_1 x_0)$. 

It is well-known that, up to a new codification of the Markov shift $\Sigma$, any locally constant potential can be expressed as a Markov potential. Because of that,  from now on we assume that $\phi(x) = \phi(x_0x_1)$ and all the proofs of the results presented here are given under that hypothesis. 

A potential $\phi$ is called {\bf coercive} when $\lim_{a \to \infty}\sup(\phi|_{[a]}) = -\infty$ and we say that $\phi$ is a {\bf summable potential} when $\sum_{a \in S} e^{\sup(\phi|_{[a]})} < \infty$.  Note that any summable potential is coercive, however the converse is not necessarily true.  Besides that, we define the {\bf first variation} of the potential $\phi$ by the expression $V_1(\phi) := \sup\{|\phi(x_0 x_1) - \phi(y_0 y_1)| :\; x_0 = y_0\}$. 

Given a metric space $Y$,  denote by $\mathrm{C}_b(Y)$, resp. $\mathcal{M}_1(Y)$, resp. $\mathcal{M}_\sigma(Y)$ the set of {\bf bounded continuous functions}, resp. {\bf Borel probability measures}, resp. {\bf Borel $\sigma$-invariant probability measures} on $Y$ and denote by $\mathcal{B}_Y$ the collection of {\bf Borel sets} on $Y$.  Besides that,  given $\mu  \in \mathcal{M}_\sigma(Y)$ we will use the notation $\mu(\phi) := \int_Y \phi d\mu$.

Consider $\mu \in \mathcal{M}_\sigma(\Sigma)$ and a potential $\phi$ such that $\mu(\phi) > -\infty$, the {\bf free energy} of $\phi$ w.r.t. $\mu$ is defined by $P_\mu(\phi) := h(\mu) + \mu(\phi)$, where $h(\mu)$ is the {\bf Kolmogorov-Sinai entropy} of the measure $\mu$.  The {\bf topological pressure} of $\phi$ is given by $P_{top}(\phi) := \sup\{P_\mu(\phi) :\; \mu \in \mathcal{M}_\sigma(\Sigma) \text{ and } \mu(\phi) > -\infty \}$ and we say that $\mu \in \mathcal{M}_\sigma(\Sigma)$ is an {\bf equilibrium state} for $\phi$ when $P_{\mu}(\phi) = P_{top}(\phi)$.

Assume that $\Sigma$ is {\bf topologically transitive} and $V_1(\phi) < \infty$, the {\bf Gurevich pressure} of the potential $\phi$ is defined as $P_G(\phi) := \limsup_{n \to \infty} \frac{1}{n} Z_n(\phi, a)$,  where $Z_n(\phi, a) := \sum_{\sigma^n x = x} e^{S_n \phi(x)}{\bf 1}_{[a]}(x)$ and $S_n \phi$ is the {\bf $n$-th ergodic sum}.  It is well-known that the former expression is independent of the choice of $a \in S$ and the limit exists when $\Sigma$ is {\bf topologically mixing}.  Moreover,  we have that $P_G(\phi) = P_{top}(\phi)$ when $\|L_\phi {\bf 1}\|_\infty < \infty$ (see \cite{MR1738951} for a detailed proof).  

Given $f \in \mathrm{C}_b(\Sigma)$,  the {\bf Ruelle operator} $L_\phi$ associated to $\phi$ calculated into the map $f$ is given by the expression $L_\phi(f)(x) := \sum_{y \in \sigma^{-1}(\{x\})} e^{\phi(y)}f(y)$.  In particular,  when $f \in \mathrm{C}_b(S)$,  the former expression is reduced to the following one $L_\phi(f)(x_0) := \sum_{\substack{a \in S \\ m_{a x_0} = 1}} e^{\phi(a x_0)}f(a)$.

It is not difficult to check that, under suitable conditions, $L_\phi$ preserves the space $\mathrm{C}_b(\Sigma)$. Actually, when $\phi$ is a summable potential with $P_G(\phi) < \infty$, there exists a unique equilibrium state $\mu_t$ associated to $t\phi$ for each $t \geq 1$ and such a probability measure is of the form $\mu_t = h_t d\nu_t$ where $L_{t\phi} h_t = e^{P_G(t\phi)} h_t$ and $L_{t\phi}^* \nu_t = e^{P_G(t\phi)} \nu_t$ (see \cite{MR2018604} and \cite{MR1738951}). 

Furthermore,  defining the {\bf transpose Ruelle operator} of the potential $\phi$ calculated in $f \in \mathcal{C}_b(S)$ by $L^\intercal_\phi(f)(x_0) := \sum_{\substack{a \in S \\ m_{x_0 a} = 1}} e^{\phi(x_0 a)}f(a)$,  it is not difficult to check that under the same hypothesis there exists a unique equilibrium state associated to the potential $t\phi^\intercal$ for each $t \geq 1$ which is of the form $\mu^\intercal_t = h^\intercal_t d\nu^\intercal_t$ where $L^\intercal_{t\phi} h^\intercal_t = e^{P_G(t\phi)} h^\intercal_t$ and $(L^\intercal_{t\phi})^* \nu^\intercal_t = e^{P_G(t\phi)} \nu^\intercal_t$.

Actually,  in \cite{MR767801} was proven that $\mu_t \in \mathcal{M}_\sigma(\Sigma)$ results in a {\bf Markov stationary probability measure} satisfying the following expression
\begin{equation}
\label{SMM}
\mu_t([\omega_0\; ...\; \omega_n]) = \pi_t(\omega_0)\prod_{i=0}^{n-1}P_t(\omega_i \omega_{i+1}) \;,
\end{equation}
where $\pi_t(a) = h_t(a)h_t^\intercal(a)$ for each $a \in S$ and $P_t(ab) := \frac{h_t^\intercal(b)}{h_t^\intercal(a)}e^{t\phi(ab) - P_G(t\phi)}$ for each $a, b \in S$ such that $m_{ab} = 1$. In particular, we also have $\sum_{a \in S}\pi_t(a) = 1$ and $\sum_{\substack{b \in \mathbb{N} \\ m_{ab} = 1}}P_t(ab) = 1$. 

On the other hand,  we say that $\mu_\infty \in \mathcal{M}_\sigma(\Sigma)$ is a {\bf ground state},  when it is an accumulation point of the family of equilibrium measures $(\mu_t)_{t \geq 1}$ at $\infty$,  which is also known as a zero temperature limit.  In fact,  under the assumption that $V_1(\phi) < \infty$ was proved existence and uniqueness of such a ground state $\mu_\infty$ (see \cite{MR3864383} and \cite{MR4637152} for details).  Furthermore,  it was proven that the accumulation point belongs to the set of {\bf maximizing measures} $\mathcal{M}_{\max}(\phi):= \{\mu \in \mathcal{M}_\sigma(\Sigma) :\; \mu(\phi) = \alpha(\phi)\}$,  where the {\bf maximizing value} is given by $\alpha(\phi) := \sup\{\mu(\phi) :\; \mu \in \mathcal{M}_\sigma(\Sigma)\}$. Besides that,  it was showed that the Kolmogorov-Sinai entropies of the family $(\mu_t)_{t \geq 1}$ satisfy $h(\mu_\infty) = \lim_{t \to \infty}h(\mu_t) = \sup\{h(\mu) :\; \mu \in \mathcal{M}_{\max}(\phi)\}$.

We say that a map $V : S \to \mathbb{R}$ is a {\bf forward sub-action} for the potential $\phi$, when for any $b \in S$ we have $V(b) = \sup\{\phi(ab) + V(a) - \alpha(\phi) :\; a \in S,\; m_{ab} = 1\}$.  On the other hand,  the map $V^\intercal : S \to \mathbb{R}$ is a {\bf backward sub-action} for $\phi$ when any $a \in S$ satisfies $V^\intercal(a) = \sup\{\phi(ab) + V^\intercal(b) - \alpha(\phi) :\; b \in S,\; m_{ab} = 1\}$.

One of the main reasons to define the so-called sub-actions is to characterize the support of the measures belonging to $\mathcal{M}_{\max}(\phi)$,  because that kind of functions act as dual objects of the maximizing measures associated to the potential.  Actually,  given an arbitrary probability measure $\mu \in \mathcal{M}_{\max}(\phi)$ and $x \in \mathrm{supp}(\mu)$,  it is not difficult to check that any forward sub-action $V$ for the potential $\phi$ satisfies the expression $V(x_{i+1}) = \phi(x_i x_{i+1}) + V(x_i) - \alpha(\phi)$ (see for instance \cite{MR2279266}).  A similar conclusion holds for the case of a backward sub-action $V^\intercal$, that is, $V^\intercal(x_i) = \phi(x_i x_{i+1}) + V^\intercal(x_{i+1}) - \alpha(\phi)$.

Moreover,  denoting by $\mathcal{V} := \{V :\; \text{$V$ is forward sub-action for $\phi$}\}$ the {\bf set of sub-actions for $\phi$},  by the Atkinson's lemma (see \cite{MR419727} for details),  we obtain that the set of {\bf non-wandering points w.r.t. $\phi$} satisfies the expression
\begin{equation*}
\Omega(\phi) = \bigcap_{V \in \mathcal{V}}(\phi + V - V \circ \sigma)^{-1}(\{\alpha(\phi)\}) \;.
\end{equation*}

It is not difficult to check that $\Omega(\phi) \subset \mathrm{supp}(\mu)$ for any $\mu \in \mathcal{M}_{\max}(\phi)$ and the set $\Omega(\phi)$ contains the {\bf omega-limit} of any point in $\Sigma$ (see \cite{MR3701349} and \cite{MR2422016} for details).  Furthermore,  given any forward sub-action $V$,  it is easy to check that the potential $\phi' := \phi - \alpha(\phi) + V - V \circ \sigma$ satisfies the condition $\phi' \leq \alpha(\phi') = 0$.

Now we are able to present the first one of the main results of this paper which proves the existence of an optimal rate function satisfying a large deviations principle.  It is important to point out that the main results of this paper are stated assuming that $\phi$ is a summable potential with $V_1(\phi) < \infty$ and $P_G(\phi) < \infty$, accepting the condition $\alpha(\phi) = 0$.  In fact,  the last one of the conditions can be easily guaranteed replacing $\phi$ by $\phi_0 := \phi - \alpha(\phi)$,  because that new potential also satisfies $V_1(\phi_0) < \infty$ and $P_G(\phi_0) < \infty$ when $\phi$ satisfies the mentioned properties.  Therefore,  to not overload the notation,  hereafter we denote such a potential by $\phi$ as well 

\begin{theorem}
\label{LDPT}
Consider the rate function $I(x) := \inf\{F_k(x) :\; k \in \mathbb{N}\}$, where
\begin{equation*}
F_k(x) := V(x_0) + V^\intercal(x_k) + \sum_{i = 0}^{k-1}\phi(x_i x_{i+1}) \;,
\end{equation*} 
for $V$ (resp. $V^\intercal$) forward (resp. backward) sub-action for $\phi$. Then,  there is a sequence $(t_n)_{n \in \mathbb{N}}$,  with $\lim_{n \to \infty} t_n = \infty$,  such that,  for any cylinder $[\omega] \subset \Sigma$ the following limit holds
\begin{equation}
\label{LDP}
\lim_{n \to \infty} \frac{1}{t_n}\log(\mu_{t_n}([\omega])) = \sup\{I(x) :\; x \in [\omega]\} \;.
\end{equation}
\end{theorem}

Now assume that $\Sigma$ has incidence matrix $M$, we call {\bf transpose shift} to the set $\Sigma^\intercal := \{x \in S^{-\mathbb{N}} :\; m_{x_{i-1}x_i} = 1\; \forall i \in -\mathbb{N}\}$ with the {\bf transpose shift map} $\sigma^\intercal((x_n)_{n \in -\mathbb{N}}):= (x_{n-1})_{n \in -\mathbb{N}}$ acting on it.  In fact, the matrix $M$ also induce the {\bf bilateral Markov shift} given by $\widehat{\Sigma} := \{x \in S^\mathbb{Z} :\; m_{x_i x_{i+1}} = 1\}$,  with the {\bf bilateral shift map} $\widehat{\sigma}((...\; x_{-1} | x_0 x_1\; ...)) := (...\; x_0 | x_1 x_2\; ...)$. In this case, we also assume that both of the sets, $\Sigma^\intercal$ and $\widehat{\Sigma}$, are equipped with the {\bf product topology} and them result in metric spaces as well.  

Given a potential $\psi : \Sigma \to \mathbb{R}$, we call $\widehat{\psi} : \widehat{\Sigma} \to \mathbb{R}$ the {\bf natural extension} of $\psi$, when $\widehat{\psi}(x) = \psi(x_0\; ...\; x_n\; ...)$ for each $x \in \widehat{\Sigma}$. Using the above, we say that a function $W : \widehat{\Sigma} \to \mathbb{R}$ is an {\bf involution kernel} for $\psi$, when the {\bf transpose potential} $\psi^\intercal : \widehat{\Sigma} \to \mathbb{R}$ given by
\begin{equation*}
	\psi^\intercal := \widehat{\psi} \circ \widehat{\sigma}^{-1} + W \circ \widehat{\sigma}^{-1} - W \;, 
\end{equation*} 
satisfies the expression $\psi^\intercal(x) = \psi^\intercal(...\; x_{-2}x_{-1})$ for each $x \in \widehat{\Sigma}$. In other words, the map $\psi^\intercal$ is completely characterized on the subshift $\Sigma^\intercal$. Moreover, note that the former expression is equivalent to have
\begin{equation}
	\label{IK}
	(\psi^\intercal + W) = (\widehat{\psi} + W) \circ \widehat{\sigma}^{-1} \;.
\end{equation}

Define $\widehat{\mu_t} \in \mathcal{M}_{\widehat{\sigma}}(\widehat{\Sigma})$ as the probability measure assigning to each cylinder $[\omega]_k := \{x \in \widehat{\Sigma} :\; x_k = \omega_0, ... , x_{k+n} = \omega_n \}$, with $n+1 = |\omega|$ and $k \in \mathbb{Z}$, the value $\widehat{\mu}_t([\omega]_k) := \mu_t([\omega_0\; ...\; \omega_n])$.  In the next theorem, we show existence of a rate function $\widehat{I} : \widehat{\Sigma} \to \mathbb{R}$ satisfying a large deviation principle on the bilateral shift $\widehat{\Sigma}$ for the family of measures $(\widehat{\mu}_t)_{t \geq 1}$.  Furthermore,  we prove that $\widehat{I}$ is a natural extension of the rate function $I$ stated in Theorem \ref{LDPT}.  The statement of the result is such as follows.

\begin{theorem}
\label{BLDPT}
Consider the rate function $I$ stated in Theorem \ref{LDPT} and assume that $I(x) > -\infty$ for any $x \in \Sigma$. Then, there is an involution kernel $W$ for the function $\phi$ such that the natural extension $\widehat{I}$ satisfies the expression 
\begin{equation*}
\widehat{I}(x) = V(x_0) + V^\intercal(x_{-1}) - W(x_{-1}x_0) + \sum_{i \in \mathbb{N}_0} (V^\intercal(x_i) - V^\intercal(x_{i-1}) + \phi^\intercal(x_{i-1}x_i)) \;,
\end{equation*}
for $V$ (resp. $V^\intercal$) a forward (resp. backward) sub-action for $\phi$.  Besides that, there is a sequence $(t_n)_{n \in \mathbb{N}}$,  with $\lim_{n \to \infty} t_n = \infty$,  such that,  for any cylinder $[\omega] \subset \widehat{\Sigma}$ the following limit holds true
\begin{equation}
\label{BLDP}
\lim_{n \to \infty} \frac{1}{t_n}\log(\widehat{\mu}_{t_n}([\omega])) = \sup\{\widehat{I}(x) :\; x \in [\omega]\} \;.
\end{equation}
\end{theorem}

\section{Large deviations principle}
\label{LDP-section}

In this section we prove that the family of stationary Markov equilibrium states $(\mu_t)_{t \geq 1}$ satisfies a large deviations principle such as appear stated in Theorem \ref{LDPT}.  In order to do that,  at first we need to prove some technical lemmas.  In fact,  for ease of computation we will use the notation $A = B \pm C$ when $B - C \leq A \leq B + C$ and $A = B^{\pm C}$ for $B^{-C} \leq A \leq B^C$.

\begin{lemma}
\label{L1L}
Consider a family of continuous functions $(f_t)_{t \geq 1}$, with $f_t : S \to \mathbb{R}$, converging in the norm $\| \cdot \|_\infty$ to a function $f : S \to \mathbb{R}$ and assume that $\sum_{a \in S} e^{tf_t(a)} < \infty$ for each $t \geq 1$. Then, the following expression holds
\begin{equation*}
\lim_{t \to \infty} \frac{1}{t}\log\Bigl( \sum_{a \in S} e^{tf_t(a)} \Big) = \sup\{f(a) :\; a \in S\} \;.
\end{equation*}
\end{lemma}
\begin{proof}
Given $\epsilon \in (0, 1)$, there exists $t_0 \geq 1$ such that $\|f - f_t\|_\infty < \epsilon$ for each $t \geq t_0$. By the above, we have
\begin{align*}
\sum_{a \in S}e^{tf(a)}
= \sum_{a \in S}e^{t(f(a) - f_t(a) + f_t(a))} 
= \sum_{a \in S}e^{t(f_t(a) \pm \epsilon)}
= e^{\pm t\epsilon} \sum_{a \in S}e^{tf_t(a)} \;.
\end{align*}

By the former expression, it follows that 
\begin{equation*}
\frac{1}{t}\log\Bigl( \sum_{a \in S} e^{tf(a)} \Big) = \epsilon \pm \frac{1}{t}\log\Bigl( \sum_{a \in S} e^{tf_t(a)} \Big) \;.
\end{equation*}

In particular, the former assumption implies that $\sum_{a \in S} e^{tf(a)} < \infty$ for each $t \geq t_0$. Besides that, it is well-known that 
\begin{equation*}
\lim_{t \to \infty} \frac{1}{t}\log\Bigl( \sum_{a \in S} e^{tf(a)} \Big) =  \sup\{f(a) :\; a \in S\}\;.
\end{equation*}

Then, taking the limit when $\epsilon$ goes to $0$, we prove our claim.
\end{proof}

\subsection{The unilateral CMS approach}
\label{ULDP-section}

Our goal throughout this section is to prove a large deviations principle in the setting of unilateral countable Markov shifts.  At first,  we show existence of a forward sub-action $V$ satisfying the conditions that appear stated in Theorem \ref{LDPT}.  Later we prove a local version of the large deviations principle in the context that we are interested in and at the end of the section we present the proof of Theorem \ref{LDPT}.

\begin{lemma}
\label{CSL}
There is $(t_n)_{n \in \mathbb{N}}$, with $\lim_{n \to \infty} t_n = \infty$,  such that,  the sequence of functions $(\frac{1}{t_n}\log(h_{t_n}))_{n \in \mathbb{N}}$ is pointwise convergent on $\Sigma$ to a map $V$.  Furthermore,  the limit function $V$ results in a forward sub-action for $\phi$. 
\end{lemma}
\begin{proof}
In \cite{MR1738951} was proven existence of a constant $K_\phi$ such that any pair $x, y \in \Sigma$ and each $a \in S$ satisfy the expression
\begin{equation*}
\bigl| e^{-nP_G(\phi)}(L^n_\phi {\bf 1}_{[a]})(x_0) - e^{-nP_G(\phi)}(L^n_\phi {\bf 1}_{[a]})(y_0) \bigr|
\leq K_\phi \;.
\end{equation*}

By the above, it follows that the sequence $\bigl( \frac{1}{n}\sum_{k=0}^{n-1}e^{-kP_G(\phi)}L^k_\phi {\bf 1}_{[a]} \bigr)_{n \in \mathbb{N}}$ is equicontinuous and uniformly bounded for each $a \in S$. Therefore, for each $x \in \Sigma$ we have $\lim_{l \to \infty}  \frac{1}{n_l}\sum_{k=0}^{n_l-1}e^{-kP_G(\phi)}(L^k_\phi {\bf 1})(x_0) = h_1(x_0)$.

Further,  by the former expressions we are able to guarantee the existence of another constant $K'_\phi$ such that for any $t \geq 1$ and each pair $x, y \in \Sigma$,  the following expression holds true
\begin{equation*}
\Bigl| \frac{1}{t}\log(h_t(x_0)) - \frac{1}{t}\log(h_t(y_0)) \Bigr| \leq K'_\phi \;. 
\end{equation*}  

The above guarantees that the family $(\frac{1}{t}\log(h_t))_{t \geq 1}$ is equicontinuous and uniformly bounded. Therefore, there is a sequence $(t_n)_{n \in \mathbb{N}}$, such that each $x \in \Sigma$ satisfies $\lim_{n \to \infty} \frac{1}{t_n}\log(h_{t_n}(x_0)) = V(x_0)$.

On the other hand,  by the variational principle in the setting of countable Markov shifts (see for instance \cite{MR1738951}),  it follows that
\begin{equation*}
\frac{P_G(t\phi)}{t} = \frac{h(\mu_t)}{t} + \mu_t(\phi) \;, 
\end{equation*}
for each $t \geq 1$.  Besides that,  since the family of entropies $(h(\mu_t))_{t \geq 1}$ is decreasing (see \cite{MR3864383} for details),  taking the limit when $t$ goes to $\infty$,  we obtain $\lim_{t \to \infty}\frac{P_G(t\phi)}{t} = \alpha(\phi) = 0$.

So,  defining $\widetilde{\phi} : S \times S \to \mathbb{R} \cup \{-\infty\}$ by
\begin{equation*}
\widetilde{\phi}(ab) :=
\begin{cases}
\phi(ab) &, m_{ab} = 1 \;; \\
-\infty &, m_{ab} = 0 \;.
\end{cases}
\end{equation*}

Since $L_{t\phi}h_t = e^{P_G(t\phi)}h_t$ for any $t \geq 1$, given $x \in \Sigma$ we obtain
\begin{align*}
\frac{P_G(t\phi)}{t} 
&= \frac{1}{t}\log(L_{t\phi} h_t)(x_0) - \frac{1}{t}\log(h_t(x_0)) \\
&= \frac{1}{t}\log\Bigl( \sum_{a \in S} e^{t(\widetilde{\phi}(ax_0) + \frac{1}{t}\log(h_t(a)))} \Big) - \frac{1}{t}\log(h_t(x_0))
\end{align*}

Then, by Lemma \ref{L1L}, it follows that
\begin{align*}
0 = \alpha(\phi) = \lim_{n \to \infty} \frac{P_G(t_n\phi)}{t_n} 
&=\sup\{\widetilde{\phi}(ax_0) + V(a) :\; a \in S\} - V(x_0) \\
&=\sup\{\phi(ax_0) + V(a) :\; a \in S,\; m_{ax_0} = 1\} - V(x_0) \;.
\end{align*}

That is,  the function $V$ results in a forward sub-action for $\phi$.
\end{proof}

\begin{remark}
\label{BCSL}
Since $L^\intercal_{t\phi} h^\intercal_t = e^{P_G(t\phi)} h^\intercal_t$ for each $t \geq 1$, following a similar procedure to the one in Lemma \ref{CSL}, it is not difficult to check that there is a sequence $(t_n)_{n \in \mathbb{N}}$ (which can be assumed as the same one stated in the lemma),  such that the sequence $(\frac{1}{t_n}\log(h^\intercal_{t_n}))_{n \in \mathbb{N}}$ converges on cylinders with the norm $\| \cdot \|_\infty$ and pointwise on $\Sigma$ to a map $V^\intercal$. Moreover, in this case the limit function $V^\intercal$ results in a backward sub-action.
\end{remark}

To prove the first one of the main results of this paper, we approximate the rate map $I$ stated in the theorem by a sequence of locally constant maps satisfying a local large deviations principle.  

It is important to remember that we are considering $\phi$ as a summable potential such that $\alpha(\phi) = 0$,  $V_1(\phi) < \infty$ and $P_G(\phi) < \infty$. The statement of the lemma is the following one.

\begin{lemma}
\label{LLDPT}
There is a sequence $(t_n)_{n \in \mathbb{N}}$ such that for each $k \in \mathbb{N}$ and any cylinder $[\omega_0\; ...\; \omega_k] \subset \Sigma$ the following limit holds true
\begin{equation}
\label{LLDP}
\lim_{n \to \infty} \frac{1}{t_n}\log(\mu_{t_n}([\omega_0\; ...\; \omega_k])) = \sup\{F_k(x) :\; x \in [\omega_0\; ...\; \omega_k]\} \;,
\end{equation}
where $F_k(x)$ satisfies the expression stated in Theorem \ref{LDPT} for $V$ (resp. $V^\intercal$) given by Lemma \ref{CSL} (resp. Remark \ref{BCSL}).
\end{lemma}
\begin{proof}
Given $k \in \mathbb{N}$ and $t \geq 1$, consider the functions $F_{k, t} : \Sigma \to \mathbb{R}$ given by the expression
\begin{equation*}
F_{k, t}(x) := \frac{1}{t}\log(h_t(x_0)) + \frac{1}{t}\log(h^\intercal_t(x_k)) + \sum_{i = 0}^{k-1}\Bigl(\phi(x_i x_{i+1}) - \frac{P_G(t\phi)}{t}\Bigr)\;.
\end{equation*}  

Then,  by Lemma \ref{CSL},  there exists a sequence $(t_n)_{n \in \mathbb{N}}$,  with $\lim_{n \to \infty} t_n = \infty$,  such that,  the sequence of functions $(F_{k, t_n})_{n \in \mathbb{N}}$ converges on cylinders in the norm $\| \cdot \|$ and pointwise on $\Sigma$ to a map $F_k$.  Moreover, for each $x \in \Sigma$ the following pointwise limit holds
\begin{equation*}
\lim_{n \to \infty} F_{k, t_n}(x) = V(x_0) + V^\intercal(x_k) + \sum_{i = 0}^{k-1}\phi(x_i x_{i+1}) =: F_k(x) \;.
\end{equation*}

On the other hand, by \eqref{SMM}, for any $t \geq 1$,  each cylinder $[\omega_0\; ...\; \omega_k]$ and any $x \in [\omega_0\; ...\; \omega_k]$,  we have the following expression
\begin{equation}
\label{SMM2}
\mu_t([\omega_0\; ...\; \omega_k]) = h_t(\omega_0)h^\intercal_t(\omega_k)e^{\sum_{i = 0}^{k-1}(t\phi(\omega_i \omega_{i+1}) - P_G(t\phi))} \;.
\end{equation}

Therefore, taking $\log$ on both sides of the former expression and multiplying by $\frac{1}{t}$, we obtain
\begin{align*}
\frac{1}{t}\log(\mu_t([\omega_0\; ...\; \omega_k])) 
&= \frac{1}{t}\log(h_t(\omega_0)) + \frac{1}{t}\log(h_t(\omega_k)) + \sum_{i = 0}^{k-1}\Bigl(\phi(\omega_i \omega_{i+1}) - \frac{P_G(t\phi)}{t}\Bigr) \\
&\leq \sup\{F_{k, t}(x) :\; x \in [\omega_0\; ...\; \omega_k]\} \;.
\end{align*}

So,  taking the limit when $n$ goes to $\infty$ on the sequence $(t_n)_{n \in \mathbb{N}}$, it follows that
\begin{equation*}
\limsup_{n \to \infty} \frac{1}{t_n}\log(\mu_{t_n}([\omega_0\; ...\; \omega_k])) \leq \sup\{F_k(x) :\; x \in [\omega_0\; ...\; \omega_k]\} \;.
\end{equation*}

On the other hand, by the definition of supremum, for any $\epsilon > 0$, there exists $x' \in [\omega_0\; ...\; \omega_k]$ such that
\begin{equation*}
\sup\{F_k(x) :\; x \in [\omega_0\; ...\; \omega_k]\} \geq F_k(x') > \sup\{F_k(x) :\; x \in [\omega_0\; ...\; \omega_k]\} - \epsilon \;.
\end{equation*}

Furthermore, since the map $F_k : \Sigma \to \mathbb{R}$ is continuous on cylinders (because such a function is uniform limit of continuous functions), there exists $m \in \mathbb{N}$, such that for any $y \in [x'_0\; ...\; x'_k\; ...\; x'_{k+m}]$,  where $x' = (x'_n)_{n \in \mathbb{N}}$,  the following inequalities hold true
\begin{equation*}
\sup\{F_k(x) :\; x \in [\omega_0\; ...\; \omega_k]\} \geq F_k(y) \geq \sup\{F_k(x) :\; x \in [\omega_0\; ...\; \omega_k]\} - 2\epsilon \;.
\end{equation*}

In particular
\begin{equation*}
\inf\{F_k(x) :\; x \in [x'_0\; ...\; x'_{k+m}]\} \geq \sup\{F_k(x) :\; x \in [\omega_0\; ...\; \omega_k]\} - 2\epsilon\;.
\end{equation*}

In addition,  also by \eqref{SMM2},  it follows that
\begin{align*}
\frac{1}{t}\log(\mu_t([\omega_0\; ...\; \omega_k])) 
&= \frac{1}{t}\log(h_t(\omega_0)) + \frac{1}{t}\log(h_t(\omega_k)) + \sum_{i = 0}^{k-1}\Bigl(\phi(\omega_i \omega_{i+1}) - \frac{P_G(t\phi)}{t}\Bigr) \\
&\geq \inf\{F_{k, t}(x) :\; x \in [\omega_0\; ...\; \omega_k]\} \;.
\end{align*}

Then,  by monotonicity of $\mu_t$,  we obtain the following expression
\begin{align*}
\frac{1}{t}\log(\mu_t([\omega_0\; ...\; \omega_k])) 
&\geq \frac{1}{t}\log(\mu_t([x'_0\; ...\; x'_{k+m}])) \\ 
&\geq \inf\{F_{k, t}(x) :\; x \in [x'_0\; ...\; x'_{k+m}]\} \;.
\end{align*}

Therefore,  taking the limit when $n \to \infty$ on the sequence $(t_n)_{n \in \mathbb{N}}$ obtained in Lemma \ref{CSL},  it follows that
\begin{align*}
\liminf_{n \to \infty} \frac{1}{t_n}\log(\mu_{t_n}([\omega_0\; ...\; \omega_k])) 
&\geq \inf\{F_k(x) :\; x \in [x'_0\; ...\; x'_{k+m}]\} \\
&\geq \sup\{F_k(x) :\; x \in [\omega_0\; ...\; \omega_k]\} - 2\epsilon\;.
\end{align*}

Since $\epsilon$ is arbitrary,  taking the limit when $\epsilon$ goes to $0$,  we conclude that the limit in expression \eqref{LLDP} holds.
\end{proof}

Now we are able to prove the first one of the main results in this paper. Actually,  in Theorem \ref{LDPT} we are able to obtain an optimal rate function satisfying the large deviations principle,  the proof of the result is such as follows.

\begin{proof}[Proof of Theorem \ref{LDPT}]
By Lemma \ref{LLDPT}, we have that for any $k \in \mathbb{N}$ and any cylinder $[\omega_0\; ..., \omega_k]$ the limit in \eqref{LLDP} holds. Besides that, since $I(x) \leq F_{k+1}(x) \leq F_k(x)$ for any $x \in \Sigma$ and each $k \in \mathbb{N}$,  it follows that
\begin{equation*}
\sup\{F_{k+1}(x) :\; x \in [\omega_0\; ...\; \omega_k]\} \leq \sup\{F_k(x) :\; x \in [\omega_0\; ...\; \omega_k]\}\;. 
\end{equation*}

In particular, when $\omega$ is an admissible word, for each $k \geq |\omega|$ we have
\begin{equation*}
\sup\{I(x) :\; x \in [\omega]\} \leq \sup\{F_k(x) :\; x \in [\omega]\} \leq \sup\{F_{|\omega|}(x) :\; x \in [\omega]\}\;.
\end{equation*}

By the former expression and the one appearing in \eqref{LLDP},  it follows that
\begin{equation*}
\sup\{I(x) :\; x \in [\omega]\} \leq \sup\{F_{|\omega|}(x) :\; x \in [\omega]\} = \lim_{n \to \infty} \frac{1}{t_n}\log(\mu_{t_n}([\omega])) \;.
\end{equation*}

On the other hand, by the properties of the supremum, given $\epsilon > 0$, there is $x' \in [\omega]$ such that
\begin{equation*}
\sup\{I(x) :\; x \in [\omega]\} \geq I(x') \geq \sup\{I(x) :\; x \in [\omega]\} - \epsilon \;.
\end{equation*}

Furthermore, by upper semi-continuity of the map $I$ (which is the supremum of continuous maps),  there exists $n_0 \in \mathbb{N}$, such that, for any $y$ belonging to the cylinder $[x'_0\; ...\; x'_{|\omega|+n_0-1}]$, we have
\begin{equation*}
\sup\{I(x) :\; x \in [\omega]\} \geq I(y) \geq \sup\{I(x) :\; x \in [\omega]\} - 2\epsilon \;.
\end{equation*}

In particular, the former expression implies
\begin{equation}
\label{Inf}
\inf\{I(x) :\; x \in [x'_0\; ...\; x'_{|\omega|+n_0-1}]\} \geq \sup\{I(x) :\; x \in [\omega]\} - 2\epsilon \;.
\end{equation}

Therefore,  by \eqref{SMM2} and the monotonicity of $\mu_t$,  for each $t \geq 1$ we obtain
\begin{equation*}
\frac{1}{t}\log(\mu_t([\omega])) \geq \frac{1}{t}\log(\mu_t([x'_0\; ...\; x'_{|\omega|+n_0-1}])) \geq \inf\{F_{k, t}(x) :\; x \in [x'_0\; ...\; x'_{|\omega|+n_0-1}]\} \;.
\end{equation*}

Then, taking the limit when $n \to \infty$ on the sequence $(t_n)_{n \in \mathbb{N}}$ obtained in Lemma \ref{CSL}, by \eqref{Inf} and the former expression,  it follows that
\begin{align*}
\lim_{n \to \infty}\frac{1}{t_n}\log(\mu_{t_n}([\omega]))
&\geq \inf\{F_k(x) :\; x \in [x'_0\; ...\; x'_{|\omega|+n_0-1}]\} \\
&\geq \inf\{I(x) :\; x \in [x'_0\; ...\; x'_{|\omega|+n_0-1}]\} \\
&\geq \sup\{I(x) :\; x \in [\omega]\} - 2\epsilon \;.
\end{align*}

So, taking the limit when $\epsilon$ goes to $0$,  we obtain
\begin{equation*}
\lim_{n \to \infty}\frac{1}{t_n}\log(\mu_{t_n}([\omega])) \geq \sup\{I(x) :\; x \in [\omega]\} \;.
\end{equation*}
\end{proof}

The next proposition establish a relation between the rate function that appears at the results presented in \cite{MR2210682} and \cite{zbMATH07528592},  and the rate function appearing at Theorem \ref{LDPT} in this paper.  The statement of the result is the following one.

\begin{proposition}
\label{IS}
Consider the map $I$ stated in Theorem \ref{LDPT} and assume that $I(x) > -\infty$ for any $x \in \Sigma$. Then, $I(x) = \sum_{i \in \mathbb{N}_0} (V(x_i) - V(x_{i+1}) + \phi(x_ix_{i+1}))$ for each $x \in \Sigma$.
\end{proposition}
\begin{proof}
At first, note that for each $k \in \mathbb{N}$, we have 
\begin{align*}
\sum_{i = 0}^{k-1} (V(x_i) - V(x_{i+1}) + \phi(x_ix_{i+1}))
&= V(x_0) - V(x_k) + \sum_{i = 0}^{k-1} \phi(x_ix_{i+1}) \\
&= -(V(x_k) + V^\intercal(x_k)) + F_k(x) \;.
\end{align*}

Therefore, taking the limit when $k \to \infty$, we obtain the following expression
\begin{equation*}
\sum_{i \in \mathbb{N}_0} (V(x_i) - V(x_{i+1}) + \phi(x_ix_{i+1})) = \lim_{k \to \infty} -(V(x_k) + V^\intercal(x_k)) + F_k(x) \;. 
\end{equation*}

Then,  it is enough to prove that $\lim_{k \to \infty} -(V(x_k) + V^\intercal(x_k))$ exists and it is equal to $0$.  

Indeed,  in \cite{MR2279266} was proven that given any $\tilde{x} \in \Omega(\phi)$,  for each $i \in \mathbb{N}_0$ we have 
\begin{equation*}
V(\tilde{x}_{i+1}) = \phi(\tilde{x}_i\tilde{x}_{i+1}) + V(\tilde{x}_i) \;,
\end{equation*}
and
\begin{equation*}
V^\intercal(\tilde{x}_i) = \phi(\tilde{x}_i\tilde{x}_{i+1}) + V^\intercal(\tilde{x}_{i+1}) \;. 
\end{equation*}

By the above,  it follows that for each $i \in \mathbb{N}_0$ the following expression holds
\begin{equation*}
V(\tilde{x}_i) + V^\intercal(\tilde{x}_i) = V(\tilde{x}_{i+1}) + V^\intercal(\tilde{x}_{i+1}) = \sup\{V(a) + V^\intercal(a):\; a \in S\}\;,
\end{equation*}
where the right side equality follows from the definitions of forward (resp.  backward) sub-action and the fact that $\tilde{x} \in \mathrm{supp}(\mu)$ for any $\mu \in \mathcal{M}_{\max}(\phi)$.

Besides that, since $\sum_{a \in S} h_t(a)h_t^\intercal(a) = \sum_{a \in S} \pi_t(a) = 1$ (see \cite{MR767801} for details), by Lemma \ref{CSL} and Remark \ref{BCSL}, we obtain that
\begin{equation*}
0 = \lim_{n \to \infty}\frac{1}{t_n} \log\Bigl( \sum_{a \in S} e^{t_n(\frac{1}{t_n}\log(h_{t_n}(a)) + \frac{1}{t_n}\log(h^\intercal_{t_n}(a)))} \Big) = \sup\{V(a) + V^\intercal(a) :\; a \in S\}\;.
\end{equation*}

On the other hand,  since we are assuming that $I(x) > -\infty$ for each $x \in \Sigma$, the following limit holds $\lim_{i \to \infty} V(x_i) - V(x_{i+1}) + \phi(x_i x_{i+1}) = 0$. Therefore, any accumulation point of the sequence $(\sigma^n(x))_{n \in \mathbb{N}}$ belongs to $\Omega(\phi)$.  In particular,  by the continuity of $V$ and $V^\intercal$,  the above implies that
\begin{equation*}
\lim_{k \to \infty} V(x_k) + V^\intercal(x_k) = \sup\{V(a) + V^\intercal(a):\; a \in S\} = 0 \;.
\end{equation*}

That is,  the proposition holds true.
\end{proof}

\subsection{The bilateral CMS approach}
\label{BLDP-section}

In this section, we present the proof of Theorem \ref{BLDPT} which show the relation between the theory developed in the previous section with the one formulated through involution kernels (see \cite{MR3377291} for details). The best of our knowledge in here is to present an extension of the large deviation principle appearing in Theorem \ref{LDPT} to the setting of bilateral countable Markov shifts using a similar approach to the ones appearing at \cite{MR2210682} and \cite{zbMATH07528592}. 

In fact,  the next lemma shows existence of an involution kernel in the setting that we are interested in.

\begin{lemma}
\label{IKL}
Consider $W : \widehat{\Sigma} \to \mathbb{R}$ given by $W(x) = W(x_{-1}x_0) := \phi(x_{-1}x_0)$. Then, the map $\phi^\intercal : \widehat{\Sigma} \to \mathbb{R}$ defined by $\phi^\intercal(x) = \phi^\intercal(x_{-2}x_{-1}) := \phi(x_{-2}x_{-1})$ satisifies the following expression
\begin{equation}
\label{IK2}
(\phi^\intercal + W)(x) = (\widehat{\phi} + W)(\widehat{\sigma}^{-1}(x)) \;.
\end{equation}
\end{lemma}
\begin{proof}
Note that the left side of the expression in \eqref{IK2} is given by
\begin{equation*}
(\phi^\intercal + W)(x) 
= \phi^\intercal(x) + W(x)
= \phi(x_{-2}x_{-1}) + \phi(x_{-1}x_0) \;.  
\end{equation*}

On the other hand, the right side of the expression in \eqref{IK2} satisfies
\begin{equation*}
(\widehat{\phi} + W)(\widehat{\sigma}^{-1}(x)) 
= \widehat{\phi}(\widehat{\sigma}^{-1}(x)) + W(\widehat{\sigma}^{-1}(x))
= \phi(x_{-1}x_0) + \phi(x_{-2}x_{-1}) \;. 
\end{equation*}

The above guarantees that the expression in \eqref{IK2} holds true, which is equivalent to have that $W$ is an involution kernel. 
\end{proof}

Now we have all the necessary tools to prove the second main result in this paper. The proof of the result is such as follows.

\begin{proof}[Proof of Theorem \ref{BLDPT}]
Consider $\widehat{I} : \widehat{\Sigma} \to \mathbb{R}$ the natural extension of the rate map $I$ stated in Theorem \ref{LDPT}. So, by Proposition \ref{IS}, we have that any $x \in \widehat{\Sigma}$ satisfies the expression
\begin{equation*}
\widehat{I}(x) = \sum_{i \in \mathbb{N}_0} (V(x_i) - V(x_{i+1}) + \phi(x_ix_{i+1}))
\end{equation*}

By the above and Lemma \ref{IKL}, it follows that
\begin{align*}
\widehat{I}(x) 
=& \lim_{k \to \infty} V(x_0) - V(x_k) + \sum_{i=0}^{k-1} \phi(x_ix_{i+1}) \\
=& \lim_{k \to \infty} V(x_0) - V(x_k) + \sum_{i=0}^{k-1} \phi^\intercal(x_{i-1}x_i) + W(x_ix_{i+1}) - W(x_{i-1}x_i) \\
=& \lim_{k \to \infty} V(x_0) - V(x_k) + W(x_{k-1}x_k) - W(x_{-1}x_0) + \sum_{i=0}^{k-1} \phi^\intercal(x_{i-1}x_i) \\
=& V(x_0) + V^\intercal(x_{-1}) - W(x_{-1}x_0) + \lim_{k \to \infty}(W(x_{k-1}x_k) - V^\intercal(x_{k-1}) - V(x_k)) \\
&+ \sum_{i \in \mathbb{N}_0} (V^\intercal(x_i) - V^\intercal(x_{i-1}) + \phi^\intercal(x_{i-1}x_i)) \;.
\end{align*}

Note that the series in the last expression above is a convergent one, because $\widehat{I}(x) > -\infty$ for each $x \in \widehat{\Sigma}$. 

Therefore, it is enough to show that the following expression holds 
\[
\lim_{k \to \infty}(W(x_{k-1}x_k) - V^\intercal(x_{k-1}) - V(x_k)) = 0 \;.
\] 

Indeed, by continuity of the maps $V$ and $V^\intercal$,  we have the following limit $\lim_{k \to \infty}(V^\intercal(x_{k-1}) + V(x_k)) = 0$. Besides that,  since each one of the accumulation points of the sequence $(\sigma^n(x))_{n \in \mathbb{N}}$ belongs to $\Omega(\phi)$ and $\alpha(\phi) = 0$,  it folllows that $\lim_{k \to \infty}W(x_{k-1}x_k) = \lim_{k \to \infty}\phi(x_{k-1}x_k) = 0$.  Then, our result holds true. 
\end{proof}

\section*{Acknowledgments} This work was supported by the Foundation for Science and Technology (FCT) Project UIDP/00144/2020. The author would like to thank to professor A. O. Lopes for his helpful comments about this paper and the University of Porto for the hospitality.


\begin{thebibliography}{30}

\bibitem[Atk76]{MR419727}
Giles Atkinson.
\newblock Recurrence of co-cycles and random walks.
\newblock {\em J. London Math. Soc. (2)}, 13(3):486--488, 1976.

\bibitem[BLL13]{MR3114331}
Alexandre Baraviera, Renaud Leplaideur, and Artur Lopes.
\newblock {\em Ergodic optimization, zero temperature limits and the max-plus
  algebra}.
\newblock Publica\c{c}\~{o}es Matem\'{a}ticas do IMPA. [IMPA Mathematical
  Publications]. Instituto Nacional de Matem\'{a}tica Pura e Aplicada (IMPA),
  Rio de Janeiro, 2013.
\newblock 29${^{{}}{\rm{o}}}$ Col\'{o}quio Brasileiro de Matem\'{a}tica. [29th
  Brazilian Mathematics Colloquium].

\bibitem[BLMV23]{MR4637152}
Elmer Beltr\'{a}n, Jorge Littin, Cesar Maldonado, and Victor Vargas.
\newblock Existence of the zero-temperature limit of equilibrium states on
  topologically transitive countable {M}arkov shifts.
\newblock {\em Ergodic Theory Dynam. Systems}, 43(10):3231--3254, 2023.

\bibitem[BLT06]{MR2210682}
Alexandre Baraviera, Artur~O. Lopes, and Philippe Thieullen.
\newblock A large deviation principle for the equilibrium states of
  {H}\"{o}lder potentials: the zero temperature case.
\newblock {\em Stoch. Dyn.}, 6(1):77--96, 2006.

\bibitem[BMP16]{BMP15}
Rodrigo Bissacot, Jairo~K. Mengue, and Edgardo P\'erez.
\newblock A large deviation principle for gibbs states on countable markov
  shifts at zero temperature.
\newblock 2016.
\newblock arXiv:1612.05831.

\bibitem[Bou01]{MR1841880}
Thierry Bousch.
\newblock La condition de {W}alters.
\newblock {\em Ann. Sci. \'{E}cole Norm. Sup. (4)}, 34(2):287--311, 2001.

\bibitem[BS03]{MR2018604}
J{\'e}r{\^o}me Buzzi and Omri Sarig.
\newblock Uniqueness of equilibrium measures for countable {M}arkov shifts and
  multidimensional piecewise expanding maps.
\newblock {\em Ergodic Theory Dynam. Systems}, 23(5):1383--1400, 2003.

\bibitem[CDLS17]{MR3656287}
L.~Cioletti, M.~Denker, A.~O. Lopes, and M.~Stadlbauer.
\newblock Spectral properties of the {R}uelle operator for product-type
  potentials on shift spaces.
\newblock {\em J. Lond. Math. Soc. (2)}, 95(2):684--704, 2017.

\bibitem[CLT01]{MR1855838}
G.~Contreras, A.~O. Lopes, and Ph. Thieullen.
\newblock Lyapunov minimizing measures for expanding maps of the circle.
\newblock {\em Ergodic Theory Dynam. Systems}, 21(5):1379--1409, 2001.

\bibitem[Con16]{MR3529118}
Gonzalo Contreras.
\newblock Ground states are generically a periodic orbit.
\newblock {\em Invent. Math.}, 205(2):383--412, 2016.

\bibitem[FV18]{MR3864383}
Ricardo Freire and Victor Vargas.
\newblock Equilibrium states and zero temperature limit on topologically
  transitive countable {M}arkov shifts.
\newblock {\em Trans. Amer. Math. Soc.}, 370(12):8451--8465, 2018.

\bibitem[Gar17]{MR3701349}
Eduardo Garibaldi.
\newblock {\em Ergodic optimization in the expanding case}.
\newblock SpringerBriefs in Mathematics. Springer, Cham, 2017.
\newblock Concepts, tools and applications.

\bibitem[GL08]{MR2422016}
E.~Garibaldi and A.~O. Lopes.
\newblock On the {A}ubry-{M}ather theory for symbolic dynamics.
\newblock {\em Ergodic Theory Dynam. Systems}, 28(3):791--815, 2008.

\bibitem[GLT09]{MR2563132}
E.~Garibaldi, A.~O. Lopes, and Ph. Thieullen.
\newblock On calibrated and separating sub-actions.
\newblock {\em Bull. Braz. Math. Soc. (N.S.)}, 40(4):577--602, 2009.

\bibitem[GT11]{MR2765475}
Eduardo Garibaldi and Philippe Thieullen.
\newblock Minimizing orbits in the discrete {A}ubry-{M}ather model.
\newblock {\em Nonlinearity}, 24(2):563--611, 2011.

\bibitem[Gur84]{MR767801}
B.~M. Gurevich.
\newblock A variational characterization of one-dimensional countable state
  {G}ibbs random fields.
\newblock {\em Z. Wahrsch. Verw. Gebiete}, 68(2):205--242, 1984.

\bibitem[Jen06]{MR2191393}
Oliver Jenkinson.
\newblock Ergodic optimization.
\newblock {\em Discrete Contin. Dyn. Syst.}, 15(1):197--224, 2006.

\bibitem[JMU05]{MR2151222}
O.~Jenkinson, R.~D. Mauldin, and M.~Urba{\'n}ski.
\newblock Zero temperature limits of {G}ibbs-equilibrium states for countable
  alphabet subshifts of finite type.
\newblock {\em J. Stat. Phys.}, 119(3-4):765--776, 2005.

\bibitem[JMU06]{MR2279266}
O.~Jenkinson, R.~D. Mauldin, and M.~Urba{\'n}ski.
\newblock Ergodic optimization for countable alphabet subshifts of finite type.
\newblock {\em Ergodic Theory Dynam. Systems}, 26(6):1791--1803, 2006.

\bibitem[Kem11]{MR2800665}
Tom Kempton.
\newblock Zero temperature limits of {G}ibbs equilibrium states for countable
  {M}arkov shifts.
\newblock {\em J. Stat. Phys.}, 143(4):795--806, 2011.

\bibitem[LMMS15]{MR3377291}
A.~O. Lopes, J.~K. Mengue, J.~Mohr, and R.~R. Souza.
\newblock Entropy and variational principle for one-dimensional lattice systems
  with a general {\it a priori} probability: positive and zero temperature.
\newblock {\em Ergodic Theory Dynam. Systems}, 35(6):1925--1961, 2015.

\bibitem[LMST09]{MR2496111}
A.~O. Lopes, J.~Mohr, R.~R. Souza, and Ph. Thieullen.
\newblock Negative entropy, zero temperature and {M}arkov chains on the
  interval.
\newblock {\em Bull. Braz. Math. Soc. (N.S.)}, 40(1):1--52, 2009.

\bibitem[Men18]{MR3779019}
Jairo~K. Mengue.
\newblock Large deviations for equilibrium measures and selection of subaction.
\newblock {\em Bull. Braz. Math. Soc. (N.S.)}, 49(1):17--42, 2018.

\bibitem[PR01]{MR1827117}
A.~A. Pinto and D.~A. Rand.
\newblock Existence, uniqueness and ratio decomposition for {G}ibbs states via
  duality.
\newblock {\em Ergodic Theory Dynam. Systems}, 21(2):533--543, 2001.

\bibitem[Sar99]{MR1738951}
Omri~M. Sarig.
\newblock Thermodynamic formalism for countable {M}arkov shifts.
\newblock {\em Ergodic Theory Dynam. Systems}, 19(6):1565--1593, 1999.

\bibitem[Var22]{zbMATH07528592}
Victor Vargas.
\newblock On involution kernels and large deviations principles on
  {{\(\beta\)}}-shifts.
\newblock {\em Discrete Contin. Dyn. Syst.}, 42(6):2699--2718, 2022.

\end{thebibliography}
\end{document}